\documentclass[letterpaper,10pt]{amsart}
\usepackage{indentfirst} 
\usepackage{amssymb}
\usepackage{amsmath}
\usepackage{mathabx} 
\usepackage{amsthm}
\usepackage{thmtools}
\usepackage{bbm}             
\usepackage[colorlinks=true]{hyperref}  
\usepackage[usenames,dvipsnames]{xcolor} 
\usepackage{tikz}

\declaretheorem{theorem}

\declaretheorem{lemma}
\declaretheorem{corollary}
\declaretheorem{proposition}

\declaretheorem{definition}

\declaretheorem{question}
\declaretheorem{remark}
\declaretheorem{claim}
\declaretheorem[name=Acknowledgements,numbered=no]{ack}

\newcommand{\R}{\mathbb{R}}
\newcommand{\Z}{\mathbb{Z}}
\newcommand{\N}{\mathbb{N}}

\def\phi{\varphi}
\def\R{{\mathbb R}}

\def\N{{\mathbb N}}
\def\Z{{\mathbb Z}}

\def\le{\leqslant}
\def\ge{\geqslant}

\begin{document}

\title[Characterization of uniform hyperbolicity]{Characterization of uniform hyperbolicity for fiber-bunched cocycles}
\date{June, 2018}

\author[R. Velozo]{Renato Velozo} \address{Facultad de Matem\'aticas,
Pontificia Universidad Cat\'olica de Chile (PUC), Avenida Vicu\~na Mackenna 4860, Santiago, Chile}
\email{ravelozo@mat.uc.cl}

\begin{abstract} We prove a new characterization of uniform hyperbolicity for fiber-bunched cocycles. Specifically, we show that the existence of a uniform gap between the Lyapunov exponents of a fiber-bunched $SL(2,\R)$-cocycle defined over a subshift of finite type or an Anosov diffeomorphism implies uniform hyperbolicity. In addition, we construct an $\alpha$-H\"{o}lder cocycle which has uniform gap between the Lyapunov exponents, however it is not uniformly hyperbolic.
\end{abstract}

\maketitle

\section{Introduction}

A linear cocycle is a pair $(T,A)$ where $T:X\to X$ is a homeomorphism defined on a compact metric space $X$ and $A:X\to GL(d,\R)$ is a continuous function. In particular, a $SL(2,\R)$-cocycle is a linear cocycle where $A$ takes values on $SL(2,\R)$. Besides, we are going to use the following notation $$A^n(x):=A(T^{n-1}x)A(T^{n-2}x)\dots A(Tx)A(x),$$ $$A^{-n}(x):=A(T^{-n}x)^{-1}A(T^{-n+1}x)^{-1}\dots A(T^{-2}x)^{-1}A(T^{-1}x)^{-1},$$ and $A^0(x):=I$ for every $x\in X$ and $n>0$. Along these notes we are going to deal specifically with uniformly hyperbolic cocycles. 
\begin{definition}
A $SL(2,\R)$-cocycle $(T,A)$ is called \emph{uniformly hyperbolic} if there are constants $c>0$ and $0<\lambda<1$ such that for every $x\in X$ there exist transverse one-dimensional spaces $E_x^s$ and $E_x^u$ in $\R^2$ such that 
\begin{enumerate}
\item $A(x)E_x^s=E_{T(x)}^s$ and $A(x)E_x^u=E_{T(x)}^u$,
\item $\|A^n(x)v^s\|\le c\lambda^n \|v^s\|$ and $\|A^{-n}(x)v^u\|\le c\lambda^n \|v^u\|$,
\end{enumerate}
for every $x\in X$, $v^s\in E_x^s$, $v^u\in E_x^u$ and $n \ge 1$. 
\end{definition}

Here $\|\cdot\|$ denotes the Euclidean norm. For $SL(2,\R)$-cocycles there is a well known characterization of uniform hyperbolicity proved by J.-C. Yoccoz in \cite{Y} (see \cite{BG} and \cite{Z} for related results).

\begin{proposition} \label{prop1}
A $SL(2,\R)$-cocycle $(T,A)$ is uniformly hyperbolic if and only if there are constants $c>0$ and $\tau>0$ such that $$\|A^n(x)\|\ge ce^{\tau n}, \quad \textrm{for all $n \ge 0$ and $x \in X$}.$$ 
\end{proposition}

In the previous proposition, $\|\cdot\|$ is the operator norm induced by the Euclidean norm. We proceed to define another main concept in these notes: Lyapunov exponents.

\begin{definition}
Let $(T,A)$ be a linear cocycle. We define the \emph{upper and lower Lyapunov exponents at a point} $x\in X$ respectively by  $$\lambda_+(x):= \lim_{n \to \infty} \frac{1}{n} \log \|A^n(x)\| \quad \textrm{and} \quad \lambda_-(x):= \lim_{n \to \infty} \frac{1}{n} \log \|A^{n}(x)^{-1}\|^{-1},$$ whenever the limits exist. 
\end{definition}

It follows from Kingman's subadditive ergodic theorem that these limits exist for every $x \in \mathcal{R}$, where $\mathcal{R}\subset X$ is a Borel set such that $\mu(\mathcal{R})=1$ for any $T$-invariant probability measure $\mu$. The elements of $\mathcal{R}$ are called \emph{regular points}. By elementary linear algebra, every periodic point is regular. For more details about Lyapunov exponents properties, see \cite{AB}. 
\\

By Proposition \ref{prop1}, for every $SL(2,\R)$-cocycle which is uniformly hyperbolic, there is a constant $\tau>0$ such that $$\lambda_+(x)= \lim_{n \to \infty} \frac{1}{n} \log \|A^{n}(x)\|\ge \tau>0, \quad \textrm{for every $x \in \mathcal{R}$.}$$ In addition, since $\|M\|=\|M^{-1}\|$ for every $M$ in $SL(2,\R)$, we have $\lambda_+=-\lambda_-$. Hence, there is a uniform gap of $2\tau$ between the Lyapunov exponents, more precisely $$\lambda_+(x)- \lambda_{-}(x)\ge 2\tau,  \quad \textrm{for every $x \in \mathcal{R}$}.$$ In the following, we are going to show that this property characterizes uniform hyperbolicity for an important class of cocycles. Before to state the result, we recall a basic definition.

\begin{definition}
A Borel set $\mathcal{S}\subset X$ is called a \emph{full probability set} if $\mu(\mathcal{S})=1$ for every $T$-invariant probability measure $\mu$.
\end{definition}

In particular the set $\mathcal{R}$ of regular points is of full probability. Furthermore, each periodic point $p=T^np$ belongs to every Borel set $\mathcal{S}\subset X$ of full probability. In fact, for the $T$-invariant measure $\mu_p$ defined by $$\mu_p=\dfrac{\delta_p+\delta_{Tp}+\dots+\delta_{T^{n-1}p}}{n},$$ the periodic point $p$ has positive measure. Now, we proceed to state the main result of these notes.

\begin{theorem} \label{theo1}
Let $(T,A)$ be a $SL(2,\R)$-cocycle defined over a transitive subshift of finite type or a transitive Anosov diffeomorphism. Suppose the cocycle satisfies the fiber-bunching condition, and there is a constant $\tau>0$ and a full probability set $\mathcal{S}\subset \mathcal{R}$ such that $$\lambda_+(x)\ge \tau \quad \textrm{for every $x \in \mathcal{S}$}.$$ Then the cocycle $(T,A)$ is uniformly hyperbolic. 
\end{theorem}

See Section \ref{section2} for the definition of fiber bunched linear cocycles. Note that \cite{C} proved a similar characterization of uniform hyperbolicity but assuming a stronger hypothesis. In fact, Y. Cao assumed a continuous invariant splitting in the tangent bundle. Furthermore, we show that the fiber-bunching condition is necessary for the validity of Theorem \ref{theo1}. More precisely, we construct a cocycle over a subshift of finite type which has uniform gap between the Lyapunov exponents, but is not uniformly hyperbolic.

\begin{ack}
I would like to thank to my supervisor J. Bochi for his continued guidance and encouragements throughout all this work. This article was supported by CONICYT Scholarship 22180035 and CONICYT PIA ACT172001.
\end{ack}

\section{Preliminaries} \label{section2}

We start recalling the definitions of subshifts of finite type and Anosov diffeomorphisms.

\begin{definition} \label{subshift} Let $Q=(q_{ij})$ be a $l\times l$ matrix with $q_{ij}\in \{0,1\}$. The matrix $Q$ is called \emph{irreducible} if for every pair $i,j\in \{1,2,\dots, l\}$, there is $m_{ij} \ge 1$ such that $(Q^{m_{ij}})_{ij} > 0$. Moreover, the \emph{subshift of finite type associated to the matrix $Q$} is a dynamical system $T:X\to X$, where $X$ is the set of sequences $$X=\{(\dots,x_{-1}| x_0,x_1,\dots)\in \{1,2,\dots, l\}^{\Z}: q_{x_n x_{n+1}}=1 \text{ for every $n\in\Z$}\},$$ and $T$ is the left-shift map defined by $T((x_n)_{n\in \Z})=(x_{n+1})_{n\in \Z}$. 
\end{definition}

Throughout this article all the subshifts will be associated to some irreducible matrix $Q$. Moreover, we are going to consider the following metric on $X$, $$d(x,y):=
\begin{cases}
2^{-N(x,y)} & \text{where } N(x,y):=\min \{|n|\ge 0: x_n\neq y_n\},\\
0   & \textrm{ if } x=y.
\end{cases}$$ Note that $(X,d)$ is a compact metric space, $T$ is a homeomorphism, and $T$ is transitive (that is, $T$ has a dense orbit).
\\
We define the local stable set of $x\in X$ by $$W_{loc}^s(x):=\{(y_n)_{n\in \Z} \in X: y_n=x_n \text{ for every } n \ge 0\},$$ and the local unstable set of $x\in X$ by $$W_{loc}^u(x):=\{(y_n)_{n\in \Z} \in X: y_n=x_n \text{ for every } n \le 0\}.$$ The global stable and unstable manifolds of $x$ are defined by $$W^s(x):=\bigcup_{n=0}^{\infty} T^{-n}(W^s_{loc}(T^nx))\quad \textrm{and} \quad W^u(x):=\bigcup_{n=0}^{\infty} T^{n}(W^u_{loc}(T^{-n}x)).$$ Note that $y \in W^s(x)$ if and only if $\lim_{n\to \infty}d(T^nx, T^ny)=0$ and $y \in W^u(x)$ if and only if $\lim_{n\to \infty}d(T^{-n}x, T^{-n}y)=0$.

\begin{definition}
Let $X$ be a connected manifold. A diffeomorphism $T:X\to X$ is called \emph{Anosov} if there is an invariant decomposition of the tangent bundle $TX$ as a direct sum of continuous $DT$-invariant sub-bundles $E^s_x$ and $E^u_x$ such that, for some appropriate Riemannian metric, $$\|DT_x(v^s)\|<\lambda<1<\lambda^{-1}<\|DT_x(v^u)\|,$$ for all $x \in X$ and for any pair of unit vectors $v^s\in E^s_x$, $v^u\in E^u_x$, where $\lambda\in (0,1)$ is a constant.
\end{definition}

Let us recall the following fundamental result on stable manifolds for an Anosov diffeomorphism. Let $d$ be the Riemannian distance function. 

\begin{theorem} \label{stablemanif}
(Stable Manifold Theorem)
Let $T: X\to X$ be an Anosov diffeomorphism of class $C^k$. Then there exist $\epsilon>0$ and $0<\lambda<1$ such that for each $x\in X$, the local stable manifold $$W^s_{loc}(x):=\{y\in X: d(T^nx, T^n y)\le \epsilon \quad \textrm{for all $n \ge 0$}\},$$ and the local unstable manifold $$W^u_{loc}(x):=\{y\in X: d(T^{-n}x, T^{-n} y)\le \epsilon \quad \textrm{for all $n \ge 0$}\},$$ are $C^k$ embedded disks tangent at $x$ to $E^s_x$ and $E^u_x$ respectively. In addition, 
\begin{itemize}
\item $T(W^s_{loc}(x))\subset W^s_{loc}(Tx)$ and $T^{-1}(W^u_{loc}(x))\subset W^u_{loc}(T^{-1}x)$;
\item $d(T(x),T(y))\le \lambda d(x,y)$ for all $y\in W^s_{loc}(x)$;
\item $d(T^{-1}(x),T^{-1}(y))\le \lambda d(x,y)$ for all $y\in W^u_{loc}(x)$;
\item $W^s_{loc}(x)$ and $W^s_{loc}(x)$ vary continuously with the point $x$ in the $C^k$ topology.
\end{itemize}

Furthermore, the global stable and unstable manifolds of $x$, $$W^s(x):=\bigcup_{n=0}^{\infty} T^{-n}(W^s_{loc}(T^nx))\quad \textrm{and} \quad W^u(x):=\bigcup_{n=0}^{\infty} T^{n}(W^u_{loc}(T^{-n}x)),$$ are smoothly immersed submanifolds of $X$ and they are characterized by $$W^s(x)=\{y\in X: \lim_{n\to \infty}d(T^nx, T^ny)=0\},$$ $$W^u(x)=\{y\in X: \lim_{n\to \infty}d(T^{-n}x ,T^{-n}y)=0\}.$$

\end{theorem}

Another property of Anosov dynamics is their local product structure. More precisely, there is a constant $\delta_1>0$ such that for every $x, y\in X$ which satisfy $d(x,y)<\delta_1$, the intersection $W^u_{loc}(x)\bigcap W^s_{loc}(y)$ consists of a unique point, denoted by $[x,y]$. Even more, the intersection $[x,y]$ is transverse and the function $[\cdot, \cdot]$ is continuous. Let us recall some other basic concepts.

\begin{definition}
A sequence $x_0,x_1,\dots, x_n=x_0$ of points is called a \emph{periodic $\epsilon$-pseudo-orbit} if $d(T(x_k),x_{k+1})<\epsilon$ for $k=0,1,\dots, n-1$.
\end{definition}

\begin{definition}
A map $T$ satisfies the \emph{closing property} if there are positive constants $C, \delta_0$ such that for $\epsilon< \delta_0$ and any periodic $\epsilon$-pseudo-orbit $(x_0,x_1,\dots,x_n)$, there is a periodic point $p$ such that $T^np=p$ and $d(T^k p, x_k)<C\epsilon$, for every $k\in \{0,1,\dots, n\}$.
\end{definition}

\begin{remark}
In particular, if a map $T$ satisfies the closing property and $x \in X$ satisfies $d(x,T^n x)<\delta_0$, then there is a periodic point $p=T^n p$ such that  the orbit segments $x,Tx, \dots, T^n x $ and $p,Tp, \dots, T^n p $ satisfy $$d(T^k x,T^k p)<C\epsilon\quad \textrm{for every $k\in \{0,1,\dots, m\}$.}$$
\end{remark}

It is well known that Anosov diffeomorphisms satisfy the closing property. Furthermore, it is trivial to check that subshifts of finite type also satisfy this property. For more details see \cite{KH}.
\\

We proceed to state a central theorem in the theory of Lyapunov exponents. 

\begin{theorem} \label{oseledets}
(Oseledets Theorem)
Let $T: X\to X$ be a $\mu$-preserving mapping and $A:X\to SL(2,\R)$ such that $\log \|A \|\in L^1(\mu)$. If $\lambda_+(x)>0$ for almost every $x\in X$, then for almost every $x\in X$ there exists a one dimensional vector space $E_x^-$ such that 
\begin{equation*}
\lim_{n\to \infty} \dfrac{1}{n}\log \|A^n(x) v\|=
\begin{cases}
\lambda_+(x) & \textrm{ if } v\in \R^2\setminus E_x^- ,\\
\lambda_-(x)  & \textrm{ if } v\in E_x^-\setminus\{0\}.
\end{cases}
\end{equation*}
Moreover, the spaces $E^-_x$ are invariant and depend measurably on the point $x$.
\end{theorem}


Note that when $(T,A)$ is a uniformly hyperbolic cocycle $E^-_x=E^s_x$ for each $x\in X$. See \cite{AB} for a detailed proof of Theorem \ref{oseledets}. See \cite{V} page 40 for the general version of Oseledets theorem for $GL(d,\R)$-cocycles. Now, we are going to define the fiber-bunching condition, the main assumption in Theorem \ref{theo1}.

\begin{definition}
Let $T:X\to X$ be either a subshift of finite type or an Anosov diffeomorphism. A linear cocycle $(T,A)$ is called \emph{fiber-bunched} if there exists $\alpha>0$ such that the function $A:X\to GL(d,\R)$ is $\alpha$-H\"{o}lder and for every $x\in X$ $$\|A(x)\|\cdot \|A(x)^{-1}\|\cdot 2^{-\alpha}<1$$ in the case where $T$ is a subshift of finite type and $$\|A(x)\|\cdot \|A(x)^{-1}\|\cdot \lambda^{\alpha}<1$$ in the case where $T$ is an Anosov diffeomorphism. We also say that the linear cocycle $(T,A)$ satisfies the \emph{fiber-bunching condition}.
\end{definition}

\begin{remark}
In our context, $A$ takes values in $SL(2, \R)$. Since $\|M\|=\|M^{-1}\|$ for every $M$ in $SL(2,\R)$, we can write the fiber bunching condition as $$\|A(x)\|^2\cdot 2^{-\alpha}<1 \quad \textrm{or} \quad \|A(x)\|^2\cdot \lambda^{\alpha}<1,$$ if the cocycle is considered over a subshift of finite type or an Anosov diffeomorphism respectively.
\end{remark}

The most useful property of fiber-bunched cocycles is the existence of holonomies. The following theorem proved by C. Bonatti, X. G\'omez-Mont and M. Viana in \cite{BGMV} (see also \cite{KS}) gives the existence of these maps and describes their main properties.

\begin{theorem} \label{theo0}
Let $(T,A)$ be a fiber-bunched linear cocycle. For every $y\in W^s(x)$, the limit $$H_{x\gets y}^{s}:=\lim_{n \to \infty}A^{n}(x)^{-1}A^{n}(y)$$ exists and defines a linear isomorphism $H^s_{x \gets y}: \R^d \to \R^d$. We say that the family of linear automorphisms $\{H_{x\gets y}^{s}: y\in W^s(x)\}$ is the \emph{stable holonomy} for the cocycle $(T,A)$. Besides, for every $y,z\in W^s(x)$
$$H^s_{x \gets x}=I, \quad H_{x\gets y}^{s}=H_{x\gets z}^{s} \cdot H_{z \gets y}^{s},$$
$$A(x)\cdot H^s_{x \gets y}=H^s_{Tx \gets Ty} \cdot A(y).$$
Also, for every $y\in W^s_{loc}(x)$, there is a positive constant $C_0$ such that $\|H_{x \gets y}^{s}-I\|\le C_0 d(x,y)^{\alpha}$. Finally, if $y\in W^u(x)$ there are analogous properties for $$H_{x \gets y}^{u}:=\lim_{n \to \infty}A^{-n}(x)^{-1}A^{-n}(y).$$
\end{theorem}

\section{Proof of Theorem \ref{theo1}}

Let us start by proving Theorem \ref{theo1} for $SL(2,\R)$-cocycles over a transitive subshift of finite type.

\begin{proof}
Let $T$ be a subshift of finite type and let $(T,A)$ be a fiber bunched $SL(2,\R)$-cocycle. Suppose $(T,A)$ is not uniformly hyperbolic. By Proposition \ref{prop1}, for all $\epsilon >0$ and $n_*\in \N$, there exist $n_0\ge n_*$ and $x=(\dots,x_{-1}|x_0,x_1,\dots)\in X$ such that $\|A^{n_0}(x)\|\le e^{\epsilon n_0}$. Since $Q$ is irreducible, there is $n_1$ depending on $x_{n_0}$ and $x_0$, such that $Q^{n_1}_{x_{n_0}x_0}>0$. Hence, there is $(c_1,c_2,\dots, c_{n_1-1})\in \{1,2,\dots, l\}^{n_1-1}$ such that $$q_{x_{n_0}c_1}=1,\quad q_{c_{n_1-1} x_0}=1,\quad \textrm{and} \quad q_{c_i c_{i+1}}=1 \quad \textrm{ for every $i\in \{1,2,\dots, n_1-2\}$}.$$ Let $p=T^{n_0+n_1}p$ a periodic point of period $n_0+n_1$, with zeroth coordinate $x_0$ such that $(p_n)_{n=0}^{n_0+n_1-1}=(x_{0},x_1,\dots, x_{n_0-1},x_{n_0},c_1,c_2,\dots, c_{n_1-1})$. Let $$y=[p,x]=(\dots x_{0},x_1,\dots, x_{n_0-1},x_{n_0},c_1,c_2,\dots, c_{n_1-1}| x_0,x_1, \dots).$$

By construction $T^{n_0}y\in W^u_{loc}(p)$ and $p\in W^u_{loc}(y)$, then $$A^{n_0}(p)=H_{p \gets T^{n_0}y}^u \cdot A^{n_0}(y) \cdot H_{y \gets p}^u.$$ Analogously, since $T^{n_0}x\in W^s_{loc}(T^{n_0}y)$ and $y\in W^s_{loc}(x)$, 

\begin{align*} 
A^{n_0+n_1}(p)&=  A^{n_1}(T^{n_0}p)\cdot A^{n_0}(p) \\ 
 &=  A^{n_1}(T^{n_0}p)\cdot H_{T^{n_0}p \gets T^{n_0}y}^u \cdot H_{T^{n_0}y \gets T^{n_0}x}^s \cdot A^{n_0}(x)\cdot H_{x \gets y}^s\cdot H_{y \gets p}^u.
\end{align*}

\begin{center}
\begin{tikzpicture}

    \draw[semithick] (1.5,4) -- (4.5,4);
    \draw[semithick] (3.8,3) -- (3.8,5);
      
    \draw [->] (3.1,3.3) -- (2.1,2.4);
    \node [above right] at (2.6,2.3) {$T^{n_0}$};

\node [above ] at (3.8,5) {$W^u_{loc}$};
\node [right] at (4.5,4) {$W^s_{loc}$};

\filldraw (3.8,4) circle (2pt);
\node [above right] at (3.8,3.6) {$y$};    


\filldraw (0.5,1) circle (2pt);
\node [below right] at (0.5,1) {$T^{n_0}(y)$};

\filldraw (0.5,1.9) circle (2pt);
\node [above right] at (0.5,1.7) {$T^{n_0}(p)$};

\node [above] at (0.5,2.2) {$W^u_{loc}$};
\node [right] at (2.3,1) {$W^s_{loc}$};

\filldraw (3.8,4.4) circle (2pt);
\node [above right] at (3.8,4.4) {$p$};

    \draw[semithick] (0.5,0) -- (0.5,2.2);
    \draw[semithick] (-1,1) -- (2.3,1);

\filldraw (0,1) circle (2pt);
\node [below] at (-0.4,1) {$T^{n_0}(x)$};

\filldraw (2,4) circle (2pt);
\node [above right] at (2,4) {$x$};
\node [below] at (2.5,-0.1) {$\textmd{Figure 1: Theorem \ref{theo1}}$};

\end{tikzpicture}   
\end{center}

If we take the norm, $$\|A^{n_0+n_1}(p)\|\le \|A^{n_1}(T^{n_0}p)\| \cdot \|H_{p \gets T^{n_0}y}^u\| \cdot \|H_{T^{n_0}y \gets T^{n_0}x}^s\| \cdot \|A^{n_0}(x)\|\cdot \|H_{x \gets y}^s\| \cdot \|H_{y \gets p}^u\|.$$ It is enough to observe that each term is bounded by a constant $C$ that does not depend on $n_0$. Note that $\|A^{n_1}(T^{n_0}p)\|$ is bounded as $n_1<\max_{1\le i,j\le n}m_{ij}<\infty$, where $m_{ij}$ are defined as before in Definition \ref{subshift}. Therefore $$\|A^{n_0+n_1}(p)\| \le C^5 \|A^{n_0}(x)\|\le C^5 e^{n_0\epsilon}$$ Hence, by submultiplicativity of the norm, $$\lambda_+(p)\le 5\frac{\log C}{n_0+n_1}+\dfrac{n_0\epsilon}{n_0+n_1}\le 2\epsilon,$$ where the previous inequality follows after choosing $n_0$ big enough. This gives a contradiction since each periodic point $p$ is in every Borel set $\mathcal{S}\subset X$ of full probability and $2\epsilon$ can be chosen less than $\tau$.
\end{proof}

We proceed to prove Theorem \ref{theo1} for a $SL(2,\R)$-cocycle defined over a transitive Anosov diffeomorphism, which satisfies the fiber-bunching condition. Firstly, we are going to state three lemmas that are going to be useful along the proof. The first lemma is a well-known result proved in \cite{BS} page 131. It justifies the transitivity hypothesis in Theorem \ref{theo1}.

\begin{lemma} \label{lemma1}
Let $T:X \to X$ be an Anosov diffeomorphism of a compact connected manifold. The following statements are equivalent:

\begin{enumerate}
\item[a)] every unstable manifold $W^u(x)$ is dense in $X$;
\item[b)] every stable manifold $W^s(x)$ is dense in $X$;
\item[c)] $T$ is topologically transitive.
\end{enumerate}
\end{lemma}

Let $d(\cdot, \cdot)$ be the distance induced by the Riemannian metric on $X$. Let $d_s(\cdot, \cdot)$ and $d_u(\cdot, \cdot)$ be the induced metrics on $W^s(x)$ and $W^u(x)$ respectively. In addition, the set $W^s_R(x)$ will denote the ball of radius $R$ centered in $x$ with respect to $d_s(\cdot, \cdot)$. The definition of $W^u_R(x)$ is analogous. Note that $T$ is a contraction with respect to $d_s$. More precisely, $$d_s(T^nx,T^ny)\le \lambda^n d_s(x,y),$$ for any $x\in X$, $y \in W^s(x)$ and $n \ge 0$. For more details see \cite{BS}.

\begin{lemma} \label{lemma2}
There is a positive constant $R_0$ such that for every pair of points $x,y\in X$, we have $W_{R_0}^u(x)\bigcap W^s_{R_0}(y)\neq \emptyset$.
\end{lemma}


The previous lemma follows directly from Lemma \ref{lemma1}, the local product structure of $T$ and the compactness of $X$.

\begin{lemma}\label{lemma3}
For every $\epsilon>0$ there exists a positive integer $N$ such that for every $n\ge N$, $x\in X$ and $z\in W_{R_0}^u(x)\bigcap W^s_{R_0}(T^{n}x)$ there is a periodic point $p=T^{n}p$ such that $d(z,p)<\epsilon.$ 

\end{lemma}

\begin{proof}
By Lemma \ref{lemma2}, for every $x\in X$, $W_{R_0}^u(x)\bigcap W^s_{R_0}(T^{n}x)\neq \emptyset$, so let $z\in W_{R_0}^u(x)\bigcap W^s_{R_0}(T^{n}x)$. By the stable manifold theorem there is a positive constant $\lambda$ such that $$d_s(T^nx, T^nz)\le \lambda^n d_s(x,z)\le R_0 \lambda^n<\epsilon$$ for every $n$ large enough. Hence, there is a positive integer $n_1$ such that $$d_s(T^nx, T^n z)<\epsilon \quad \textrm{for every $x\in X$ and $n\ge n_1$.} $$ Analogously there is a positive integer $n_2$ such that $$d_u(T^{-n}z, T^{-n} (T^{n_0}x))<\epsilon \quad \textrm{for every $x\in X$ and $n\ge n_2$.} $$ Let $\hat{n}=\max\{n_1,n_2\}$, for $n>2\hat{n}$ we can consider the periodic $\epsilon$-pseudo-orbit $\{x_k\}_{k=1}^{n}$ defined by $x_i=T^iz$ if $i\in \{0,1,\dots \hat{n}-1\}$, $x_{i}=T^{i}x$ if $i\in \{\hat{n},\hat{n}+1,\dots n-\hat{n}-1\}$, $x_{i}=T^{-(n-i)}z$ if $i\in \{n-\hat{n},n-\hat{n}+1,\dots n\}$. Graphically, $$z\mapsto Tz\mapsto \dots\mapsto T^{\hat{n}-1}z\mapsto T^{\hat{n}}x\mapsto T^{\hat{n}+1}x\mapsto \dots$$ $$\dots\mapsto T^{n-\hat{n}-1}x\mapsto T^{-\hat{n}}z\mapsto T^{-\hat{n}+1}z\mapsto \dots \mapsto z.$$

\begin{center}
\begin{tikzpicture}

    \draw[semithick] (-1,4) -- (5,4);
    \draw[semithick] (4,0) -- (4,5);

    \draw[semithick] (-0.5,3) -- (1.3,3);
    \draw[semithick] (2.6,2.4) -- (2.6,0.8);
    
\filldraw (2.6,2) circle (2pt);
\node [above right] at (1.5,1.7) {$T^{-\hat{n}}z$};

\filldraw (2.6,1.2) circle (2pt);
\node [above right] at (0.7,0.8) {$T^{-\hat{n}}(T^{n}x)$};

\filldraw (0,3.0) circle (2pt);
\node [above right] at (-0.4,2.4) {$T^{\hat{n}}x$};

\filldraw (1,3.0) circle (2pt);
\node [above right] at (0.6,2.4) {$T^{\hat{n}}z$};

\node [above right] at (3.4,5) {$W_{R_0}^u(T^{n}x)$};
\node [above right] at (5,3.7) {$W_{R_0}^s(x)$};

        \draw [->] (0.3, 3.8) -- (0.3,3.2);
        \draw [->] (2.8, 1.6) -- (3.8,1.6);
        \draw [->] (0.3, 2.4) -- (0.9,1.6);
\node [above right] at (0.35,3.3) {$T^{\hat{n}}$};
\node [above right] at (3,1.6) {$T^{\hat{n}}$};

\filldraw (4,4) circle (2pt);
\node [above right] at (4,4) {$z$};

\filldraw (4,1) circle (2pt);
\node [below left] at (5.3,1) {$T^{n}(x)$};

\filldraw (3.8,4.3) circle (2pt);
\node [below left] at (3.8,4.5) {$p$};

\filldraw (0,4) circle (2pt);
\node [above right] at (0,4) {$x$};
\node [below] at (2.6,-0.1) {$\textmd{Figure 2: Lemma \ref{lemma3}}$};

\end{tikzpicture}   
\end{center}

The previous inequalities imply that $d(T(T^{\hat{n}-1}z), T^{\hat{n}}x)<\epsilon$ and $d(T(T^{n-\hat{n}-1}x), T^{-\hat{n}}z)<\epsilon,$ hence $\{x_k\}_{k=1}^{n}$ is a periodic $\epsilon$-pseudo-orbit. By the Anosov closing lemma there is a periodic point $p=T^{n}p$ such that $d(T^k p, x_k)<C\epsilon$ for every $k\in \{0,1,\dots, m\}$. In particular $d(z,p)<\epsilon$. Hence, it is enough to consider $N=2\hat{n}+4$.
\end{proof}



Finally, we go on with the proof of Theorem \ref{theo1} for a $SL(2,\R)$-cocycle defined over a transitive Anosov diffeomorphism.

\begin{proof}
Let $T$ be an Anosov diffeomorphism and let $(T,A)$ be a fiber bunched $SL(2,\R)$-cocycle. Suppose $(T,A)$ is not uniformly hyperbolic. By Proposition \ref{prop1}, for all $\epsilon >0$ and $n_*\in \N$, exist $n_0\ge n_*$ and $x\in X$ such that $\|A^{n_0}(x)\|\le e^{\epsilon n_0}$. Along this proof we are going to consider stable manifolds of size $R_0$, where $R_0$ comes from Lemma \ref{lemma2}. In the following, we choose $z\in W_{R_0}^u(x)\bigcap W^s_{R_0}(T^{n_0 }x)$ which exists by Lemma \ref{lemma2}. By Lemma \ref{lemma3} there is a point $p$ such that $d(z,p)<\delta_1$, where $\delta_1>0$ is such that for every $x, y\in X$ the intersection $W^u_{loc}(x)\bigcap W^s_{loc}(y)$ is well defined when $d(x,y)<\delta_1$. Let us define $y=[p,z]\in W^u_{loc}(p)\bigcap W^s_{loc}(z)$. Note the expression $$H_{p \gets T^{n_0}y}^u \cdot H_{T^{n_0}y \gets T^{n_0}x}^s \cdot A^{n_0}(x)\cdot H_{x \gets y}^s\cdot H_{y \gets p}^u$$ is well defined and equals to $A^{n_0}(p)$. Then \begin{align*}\|A^{n_0}(p)\|&\le \|H_{p \gets T^{n_0}y}^u\| \cdot \|H_{T^{n_0}y \gets T^{n_0}x}^s\| \cdot \|A^{n_0}(x)\|\cdot \|H_{x \gets y}^s\| \cdot \|H_{y \gets p}^u\|\\
&\le \|H_{p \gets y}^u\|\cdot \|H_{y \gets T^{n_0}y}^u\| \cdot \|H_{T^{n_0}y \gets T^{n_0}x}^s\| \cdot \|A^{n_0}(x)\|\cdot \|H_{x \gets y}^s\| \cdot \|H_{y \gets p}^u\|.\end{align*} To conclude the proof it is enough to note that each term is bounded by a constant $C$ depending on the size of the unstable and stable manifolds under consideration. The only term which is not clearly bounded is $\|H_{y \gets T^{n_0}y}^u\|$. In order to bound this term, we state the following lemma which follows directly from the continuity of the stable manifolds.

\begin{lemma}
Let $x,y\in X$. For all $R_0>0$ there exists $\epsilon_1$ less than the large of the local stable and unstable manifolds such that if $y=W_{R_0}(x)$ and $y'\in W^s_{\epsilon_1}(y)$ then there is a unique point $x'\in X$ such that $W^u_{R_0+2\epsilon_0}(y')\bigcap W^s_{\epsilon_0}(x)=\{x'\}.$
\end{lemma}

Applying the previous lemma to $y:=z$, $x:=T^{n_0x}$, $y':=y$ and $x':=T^{n_0}y$ we get a bound for $\|H_{y \gets T^{n_0}y}^u\|$ depending on $R_0$ and the large of the local stable and unstable manifolds. Therefore $$\|A^{n_0}(p)\| \le C^5 \|A^{n_0}(x)\|\le C^5 e^{n_0\epsilon}.$$ Hence, by submultiplicativity of the norm, $$\lambda_+(p)\le 5\frac{\log C}{n_0}+\epsilon\le 2\epsilon,$$ where the previous inequality follows after choosing $n_0$ big enough. This gives a contradiction since each periodic point $p$ is in every Borel set $\mathcal{S}\subset X$ of full probability and $2\epsilon$ can be chosen less than $\tau$.
\end{proof}

\begin{remark}
More precisely, we showed that a $SL(2,\R)$-cocycle over a transitive subshift of finite type or a transitive Anosov diffeomorphism is uniformly hyperbolic if and only if it has uniform gap for every invariant measure supported on a periodic orbit. Nevertheless, this is not surprising since B. Kalinin proved in \cite{K} that the Lyapunov exponents of a linear cocycle $(T,A)$ can be arbitrarily approximated by Lyapunov exponents of a measure supported on a periodic orbit. As a result the cocycle $(T,A)$ has uniform gap for every $T$-invariant measure if and only if it has uniform gap for every invariant measure supported on a periodic orbit.
\end{remark}

\begin{remark}
The previous proof works identically for a cocycle over a hyperbolic homeomorphism. It is not necessary to consider a cocycle over an Anosov diffeomorphism to get the result. See \cite{Sak} for more details on hyperbolic homeomorphisms.
\end{remark}

%
    







\section{Counterexample} \label{example}

In the following, we are going to exhibit a cocycle which has uniform gap between the Lyapunov exponents in a set of full probability $\mathcal{S}$, however it is not uniformly hyperbolic. In particular, it cannot satisfy the fiber-bunching condition. This example shows that the fiber-bunching condition is necessary in Theorem \ref{theo1}. See \cite{CLR} and \cite{G} for more complex constructions of cocycles with similar properties. 
\\
\\
Let $X=\{0,1\}^{\Z}$ and $T:X\to X$ the left shift map. We consider a cocycle $A: X\to SL(2,\R)$ defined by 
$$A(x):= \left(\begin{array}{clc}
2 & 0 \\
0 & 1/2
\end{array}\right)R_{\theta(x)}, $$ 
where the function $R_{\theta(x)}$ is a rotation of angle $\theta(x)$. Let $V:=\{x\in X: x_0=1\}$ be a neighbourhood of $q:=(\dots,0,0,1,0,0,\dots)$. Let us define $\theta$ as  
\begin{equation*}
\theta(x):=
\begin{cases}
\pi/2 & \textrm{ if } x = q ,\\
\pi/2 -2^{-k(x)/8}  & \textrm{ if } x \in V\setminus \{q\} \textrm{ and } k(x)> k_0  ,\\
0  & \textrm{ if } x =\vec{0} \textrm{ or } k(x)\le k_0,
\end{cases}
\end{equation*}
where $k(x):=\min \{|n|; n\neq 0, x_n=1\}$ and $k_0$ is a positive integer which will be defined later. Note that $\theta(x)\in [0,\pi/2]$ for every $x\in X$. Also, we observe that when $x$ tends to $q$, $k(x)$ tends to infinity, hence $\theta(x)$ tends to $\pi/2$. In particular, $A$ is continuous as required. More precisely, we proceed to prove the following theorem. 

\begin{theorem} The $SL(2,\R)$-cocycle $(T,A)$ defined above has the following properties:
\begin{enumerate}
\item The cocycle $(T,A)$ is not uniformly hyperbolic.

\item There is a set of full probability $\mathcal{S}$, such that $\lambda_+(x)\ge \log2/2>0$ for every $x\in \mathcal{S}$.
\end{enumerate}
\end{theorem}

\begin{claim}
The $SL(2,\R)$-cocycle $(T,A)$ is not uniformly hyperbolic.
\end{claim}

\begin{proof}
Since $$\lim_{n\to \pm \infty}T^nq=\vec{0}:=(\dots,0,0,0,\dots)=T(\vec{0}) \quad \textrm{and} \quad  R_{\theta(q)}=R_{\pi/2},$$ by definition $q$ is a homoclinic point for the fixed point $\vec{0}$, therefore the cocycle cannot be uniformly hyperbolic. Let us suppose that $(T,A)$ is uniformly hyperbolic, by the invariance of $E^s_x$ $$E^s_{T^nq}=A^{2n}(T^{-n}q)E^s_{T^{-n}q}=\left(\begin{array}{clc}
2 & 0 \\
0 & 1/2
\end{array}\right)^nR_{\pi/2}\left(\begin{array}{clc}
2 & 0 \\
0 & 1/2
\end{array}\right)^n E^s_{T^{-n}q}$$ as $q$ is the only point of $V$ in the orbit of $q$. Besides, by the continuity of $E^s_x$, we have $E^s_{T^nq}\thickapprox E^s_{\vec{0}}=\{x=0\}$ for $n$ big enough. Hence, we would have that $E^s_{T^{-n}q}\thickapprox \{y=0\}$. However, by the continuity of $E^s_x$, we have $E^s_{T^{-n}q}\thickapprox E^s_{\vec{0}}=\{x=0\}$, a contradiction. 
\end{proof}

\begin{claim}
There is a set of full probability $\mathcal{S}$, such that $\lambda_+(x)\ge \log 2/2>0$ for every $x\in \mathcal{S}$.
\end{claim}

In the following, we are going deal with a linear cocycle induced by $(T,A)$ and the neighbourhood $V$, which allows us to prove the gap between the Lyapunov exponents. Let $V_0:=\bigcap_{k=1}^{\infty} \bigcup_{n=k}^{\infty} T^{-n}(V)\bigcap V$ the set of points in $V$ which return infinitely many times to $V$. Let $T_V:V_0\to V_0$ be the first return map defined by $$T_V(x):=T^{N_V(x)}(x), \quad \textrm{where $N_V(x):=\inf \{n \ge 1: T^n(x)\in V_0 \}$.}$$ Let $A_V:V\to SL(2,\R)$ be the function defined by $A_V(x):=A^{N_V(x)}(x)$. We proceed to prove the key lemma in order to prove the gap between the Lyapunov exponents.


\begin{remark} During the proof of the following lemma we are going to use that $k(x)\le N_V(x)$ and $k(T_V(x))\le N_V(x)$ for every $x\in X$.
\end{remark}

We say a set $\mathcal{C}\subset \R^2$ is a \emph{cone} if it is a homogeneous space between two transverse one-dimensional spaces. In the following lemma we prove the existence of a family of invariant cones $\mathcal{C}(x)$ for each $x\in V_0$.

\begin{lemma} \label{lemma4}
For every $x\in V_0$, there is a cone $\mathcal{C}(x)\subset \R^2$ such that $$A_V(x)\mathcal{C}(x)\subset \mathcal{C}(T_Vx).$$ Moreover, for every unit vector $v\in \mathcal{C}(x)$, we have $\|A_V(x)v\|\ge 2^{N_V(x)/2}$. 
\end{lemma}

\begin{proof}
By definition $$A_V(x)= \left(\begin{array}{clc}
2 & 0 \\
0 & 1/2
\end{array}\right)^{N_V(x)}R_{\theta(x)}. $$ Let us define $\beta(x)=2^{-k(x)/2+1/4}$. Note that $$0<\dfrac{\pi}{2}-\theta(x)-\beta(x)<\dfrac{\pi}{2}-\theta(x)+\beta(x)<\dfrac{\pi}{2}.$$ By our definition of $\beta(x)$, $\pi/2>\theta(x) +\beta(x)$ follows directly. Besides, $\theta(x)>\beta(x)$ is equivalent to $$\dfrac{\pi}{2}>2^{-k(x)/8}+2^{-k(x)/2+1/4}.$$ If $k(x)$ is big enough, $2^{-k(x)/8}<0.3$. Hence $$2^{-k(x)/8}+2^{-k(x)/2+1/4}< 0.3+2^{1/4}<\dfrac{\pi}{2}.$$ Due to last condition, it makes sense define the cone $$\mathcal{C}(x)=\R^2 \setminus \Big\{(x,y)\in \R^2: \dfrac{\pi}{2}-\theta(x)-\beta(x)\le \arctan \Big(\dfrac{y}{x}\Big)\le \dfrac{\pi}{2}-\theta(x)+\beta(x) \Big\}$$ showed on Figure 2. We proceed to prove that the cone $\mathcal{C}(x)$ satisfies the lemma. Let $v$ be a unit vector in $\mathcal{C}(x)$. If $k(x)$ is big enough, then $\sin \beta(x)>2^{-1/8}\beta(x)$. Hence $$\|A_V(x)v\|\ge 2^{N_V(x)}\sin \beta(x)\ge 2^{N_V(x)-k(x)/2+1/8} \ge 2^{N_V(x)/2+1/8}\ge 2^{N_V(x)/2}.$$

\begin{center}
\begin{tikzpicture}  
 
 \draw[draw=gray!50!white,fill=gray!50!white] 
    plot[smooth,samples=100,domain=-6:-4] (\x,{-sqrt(1-(\x+5)^2)}) -- 
    plot[smooth,samples=100,domain=-4:-6] (\x,{ sqrt(1-(\x+5)^2) });

 \draw[draw=white!50!white,fill=white!50!white] 
    plot[smooth,samples=100,domain=-7:-3] (\x,{(\x+5)/7}) -- 
    plot[smooth,samples=100,domain=-3:-7] (\x,{(\x+5)/2});
    
 \draw[draw=white!50!white,fill=white!50!white] 
    plot[smooth,samples=100,domain=-7:-3] (\x,{(\x+5)/2}) -- 
    plot[smooth,samples=100,domain=-3:-7] (\x,{(\x+5)});
    
\node [below] at (-3.8,-1) {$\mathcal{C}(x)$};    
\node [below] at (-3.5,0.72) {\footnotesize$\beta(x)$};    
    
\node [below] at (-0,-2.6) {$\textmd{Figure 3: Lemma \ref{lemma4}}$};

 \draw[draw=gray!50!white,fill=gray!50!white] 
    plot[smooth,samples=100,domain=-1:1] (\x,{-sqrt(1-(\x)^2)}) -- 
    plot[smooth,samples=100,domain=1:-1] (\x,{sqrt(1-(\x)^2)});
    
 \draw[draw=white!50!white,fill=white!50!white] 
    plot[smooth,samples=100,domain=0:0.4] (\x,{3*(\x)}) -- 
    plot[smooth,samples=100,domain=0.4:0] (\x,{1.2});

 \draw[draw=white!50!white,fill=white!50!white] 
    plot[smooth,samples=100,domain=-0.4:0] (\x,{-3*(\x)}) -- 
    plot[smooth,samples=100,domain=0:-0.4] (\x,{1.2});
    
 \draw[draw=white!50!white,fill=white!50!white] 
    plot[smooth,samples=100,domain=0:0.4] (\x,{-1.2}) -- 
    plot[smooth,samples=100,domain=0.4:0] (\x,{-3*(\x)});

 \draw[draw=white!50!white,fill=white!50!white] 
    plot[smooth,samples=100,domain=-0.4:0] (\x,{-1.2}) -- 
    plot[smooth,samples=100,domain=0:-0.4] (\x,{3*(\x)});
   
 \node [below] at (1.3,-0.7) {$R_{\theta}\mathcal{C}(x)$};

    \draw [<->] (-7.3, 0) -- (-2.3,0);
    \draw [<->] (-5, -2.2) -- (-5,2.4);

    \draw [domain=-0.7:0.75,samples=100] plot(\x, { 3*(\x) } );
    \draw [domain=-0.75:0.7,samples=100] plot(\x, { -3*(\x) } );
	\node [below] at (0.35,2.2) {\footnotesize$\beta(x)$};    
    \draw [domain=0:0.7,samples=100] plot(\x, { sqrt(4.9-(\x)^2) } );

    \draw [<->] (-2.1, 0) -- (2.3,0);
    \draw [<->] (0, -2.2) -- (0,2.4);

 \draw[draw=gray!50!white,fill=gray!50!white] 
    plot[smooth,samples=100,domain=5:7] (\x,{-(\x-5)/11}) -- 
    plot[smooth,samples=100,domain=7:5] (\x,{(\x-5)/11});

 \draw[draw=gray!50!white,fill=gray!50!white] 
    plot[smooth,samples=100,domain=5:3] (\x,{(\x-5)/11}) -- 
    plot[smooth,samples=100,domain=3:5] (\x,{-(\x-5)/11});

    \draw [<->] (2.7, 0) -- (7.6,0);
    \draw [<->] (5, -2.2) -- (5,2.4);
    
  \node [below] at (7,-0.5) {$A_V(x)\mathcal{C}(x)$};

    \draw [domain=-7.2:-2.8,samples=100] plot(\x, { (\x+5)/2 } );
    \draw [domain=-7.1:-2.9,samples=100] plot(\x, { (\x+5)/7 } );
    \draw [domain=-6.8:-3.4,samples=100] plot(\x, { (\x+5) } );

    \draw [domain=2.9:7.3,samples=100] plot(\x, { (\x-5)/11 } );
    \draw [domain=2.9:7.3,samples=100] plot(\x, { -(\x-5)/11 } );
    \draw [domain=-3.24:-3.045,samples=100] plot(\x, { sqrt(3.9-(\x+5)^2) } );

 \draw [domain=-3:-2.76393202250021,samples=100] plot(\x, { sqrt(5-(\x+5)^2) } );

        \node [below] at (-2.1, 0.9) {\footnotesize$\pi/2-\theta(x)$};

    \draw [->] (2, 1.4) -- (3,1.4);
    \draw [->] (-3.1, 1.4) -- (-2,1.4);

 \node [above] at (-2.5,1.4) {$R_{\theta}$};    

 \node [above] at (3,1.4) {\footnotesize$\left(\begin{array}{clc}
2 & 0 \\
0 & 1/2
\end{array}\right)^{N_V(x)}$};    

 \node [above] at (7.4,-0.09) {\footnotesize$\gamma$};    

    \draw [domain=7.227:7.23606797749979,samples=100] plot(\x, { sqrt(5-(\x-5)^2) } );

\end{tikzpicture}   
\end{center}

Let $\gamma$ be the greatest angle between a vector in $A_V(x)\mathcal{C}(x)$ and the $x$ axis, which is in the first quadrant. Consequentially, it is enough to prove $$\gamma(x)<\dfrac{\pi}{2}-\theta(T_Vx)-\beta(T_Vx) \quad \text{for all $x \in V_0$}$$ to get the invariant condition. We can assume the vectors $$(\cos \gamma(x), \sin \gamma(x))\quad \text{and} \quad  (2^{N_V(x)}\sin \beta(x), 2^{-N_V(x)}\cos \beta(x)),$$ are linearly dependent. If $k(x)$ is big enough, $\cot \beta(x)\le 2^{1/8}\beta(x)^{-1}=2^{k(x)/2-1/8}$. Hence $$\tan \gamma(x)=2^{-2N_V(x)}\cot \beta(x) \le 2^{-2N_V(x)+k(x)/2-1/8}\le 2^{-3N_V(x)/2-1/8},$$ which tends to zero when $k(x)$ tends to infinite. As $\gamma(x)\in (0, \pi/2)$, we conclude that $\gamma(x)$ tends to zero. By the previous calculation, $$\gamma(x) \le 2^{1/16}\tan \gamma(x)\le 2^{-3N_V(x)/2-1/16}\le 2^{-3k(T_Vx)/2-1/16},$$ if $k(x)$ is big enough. Finally we show that $$\dfrac{\pi}{2}-\theta(T_Vx)-\beta(T_Vx)=2^{-k(T_Vx)/8}-2^{-k(T_Vx)/2+1/4}$$ $$=2^{-k(T_Vx)/2}(2^{3k(T_Vx)/8}-2^{1/4}) \ge 2^{-k(T_Vx)/2} \ge 2^{-3k(T_Vx)/2-1/16}$$

Last two inequality series prove that $A_V(x)\mathcal{C}(x)\subset \mathcal{C}(T_Vx)$ when $k(x)>k_0$ for some big enough positive integer $k_0$.



\end{proof}

Finally, let $\mu$ be an ergodic $T$-invariant measure. If $\mu(V)=0$, then $\theta(x)=0$ and consequently $\lambda_+(x)=\log 2$ $\mu$-almost everywhere. Otherwise, by Poincar\'e recurrence theorem $\mu(V)=\mu(V_0)$ for every $T$-invariant probability measure $\mu$. Let us define $\mathcal{C}_{\infty}(x):=\bigcap_{n=0}^{\infty} A^n_V(T^{-n}_V x)\mathcal{C}(T^{-n}_Vx).$ Note that $\mathcal{C}_{\infty}(x)$ has nonzero vectors since $A^{n+1}_V(T^{-(n+1)}_V x)\mathcal{C}(T^{-(n+1)}_Vx)\subset A^n_V(T^{-n}_V x)\mathcal{C}(T^{-n}_Vx)$ for every $n\in \N$. By Lemma \ref{lemma4} $$\|A_V^n(x)v\|\ge 2^{N_V(T^{n-1}_Vx)}\|A_V^{n-1}(x)v\|\ge 2^{N_V(T^{n-1}_Vx)+N_V(T^{n-2}_Vx)+\dots+N_V(x))/2}\|v\|,$$ for every $x\in V_0$ and $v\in \mathcal{C}_{\infty}(x)$. Hence, there is a sequence $j_1<j_2<\dots<j_n=N_V(T^{n-1}_Vx)+N_V(T^{n-2}_Vx)+\dots+N_V(x)$ such that $\|A^{j_n}(x)\cdot v \|\ge 2^{j_n/2}\|v\|$ for every $v\in \mathcal{C}_{\infty}$ and $n\in \N$. Consequently, by Oseledets's theorem $$\lambda_{+}(x)= \lim_{n\to \infty} \dfrac{\log \|A_V^n(x)v\|}{n}= \lim_{n\to \infty} \dfrac{\log \|A_V^{j_n}(x)v\|}{j_n}\ge \dfrac{\log 2}{2}$$ for almost every $x\in V_0$ and $v\in \mathcal{C}_{\infty}(x)$. It proves the gap between the Lyapunov exponents of the $SL(2,\R)$-cocycle $(T,A)$ for an arbitrary ergodic measure $\mu$. Furthermore, by the ergodic decomposition theorem we get that $\lambda_+(x) \ge \log 2/2$ for every $x\in V_0$ and every $T$-invariant measure $\mu$. Note that when $k(x)\le k_0$ the cocycle does not have rotations, and the Lyapunov exponent $\lambda_+$ is equal to $\log 2$. As a result $$\lambda_+(x)\ge \dfrac{\log 2}{2} \quad \text{for all $x \in \mathcal{S}$},$$ where $\mathcal{S}:=[\bigcup_{n\in \Z}T^n(X\setminus V)\bigcup V_0]\bigcap \mathcal{R}$ is a set of full probability. 

\begin{remark}
The cocycle is $1/8$-H\"{o}lder continuous as $2^{-k(x)}=d(x,q)$ and $\theta(x)=\pi/2 -2^{-k(x)/8}$ if $x \in V\setminus \{q \}$ and $k(x)> k_0$. We notice directly that the cocycle does not satisfy the fiber-bunching condition. Generally, a $SL(2,\R)$-cocycle which is $\alpha$-H\"{o}lder satisfies the fiber-bunching condition if and only if $\|A(x)\|^2\cdot 2^{-\alpha}<1$. However, last example satisfies $\|A(x)\|=2$ and $\alpha=1/8$, so the cocycle does not satisfy the fiber-bunching condition.
\end{remark}

Using a similar strategy, one should be able to construct similar examples where $T$ is an Anosov diffeomorphism.

\section{Final Remarks}

In this section we are going to state a natural question which is motivated by Theorem \ref{theo1}. In order to do it, we are going to define a well known concept called dominated splitting, for more details see \cite{Sam}. In the following, $(T,A)$ will denote a $GL(d,\R)$-cocycle. In addition, $\sigma_1(M)\ge \dots \ge \sigma_{d}(M)$ will be the singular values of a matrix $M$ and $\mathfrak{m}(M)=\inf_{\|v\|=1} \|Mv\|$ will be the co-norm of a matrix $M$. Note that $\sigma_1(M)=\|M\|$ and $\mathfrak{m}(M)=\|M^{-1}\|^{-1}=\sigma_d(M)$ for every $M\in GL(d,\R)$. We proceed with the definition of a dominated splitting.

\begin{definition}
We say that $A$ admits a \emph{dominated splitting of index $i$} if there is a $A$-invariant splitting $V=E\oplus F$ where $dim E=i$ and there are constants $C>0$ and $0<\tau<1$ such that $$\dfrac{\|A^n(x)|F_x\|}{m(A^n(x)|E_x)}<C\tau^n \quad \textrm{for every $x\in X$ and every $n\ge 0$.}$$
\end{definition}

\begin{remark} Note that for $SL(2,\R)$-cocycles, last definition is equivalent to uniform hyperbolicity.
\end{remark}

The following theorem proved by J. Bochi and N. Gourmelon in \cite{BG} generalizes Proposition \ref{prop1} to higher dimensions.

\begin{theorem} \label{theo3} The following assertions about a linear cocycle are equivalent 
\begin{itemize}
\item[a)] There is a dominated splitting of index $i$.
\item[b)] There exist $C>0$ and $\tau<1$ such that $\dfrac{\sigma_{i+1}(A^n(x))}{\sigma_{i}(A^n(x))}<C\tau^n$ for all $x\in X$ and $n\ge 0.$
\end{itemize}
\end{theorem}

In fact, a way to define the intermediate Lyapunov exponents is through singular values. More precisely, we have $$\lim_{n\to \infty} \dfrac{1}{n} \log (\sigma_i(A^n(x)))=\lambda_i(x)$$ for every $x\in \mathcal{S}$, where $\mathcal{S}$ is a full probability set. In particular, $\lambda_+=\lambda_1$ and $\lambda_-=\lambda_d$, since $\sigma_1(M)=\|M\|$ and $\|M^{-1}\|^{-1}=\sigma_d(M)$ for every $M\in GL(d,\R)$. Hence, the existence of a dominated splitting of index $i$ implies the uniform gap between $\lambda_i(x)$ and $\lambda_{i+1}(x)$. More precisely, $$\lambda_i(x)-\lambda_{i+1}(x)\ge \log (\tau^{-1}) \quad \textrm{for all $x\in \mathcal{S}$}.$$ Firstly, we state a direct consequence of the proof of Theorem \ref{theo1} and Theorem \ref{theo3}.

\begin{corollary} 
Let $(T,A)$ be a $GL(2,\R)$-cocycle defined over a transitive Anosov diffeomorphism, which satisfies the fiber-bunching condition. Then, if there is a constant $\tau>0$ and a full probability set $\mathcal{S}\subset \mathcal{R}$ such that $$\lambda_+(x)- \lambda_-(x) \ge \tau \quad \textrm{for all $x \in \mathcal{S}$},$$ the cocycle $(T,A)$ admits a dominated splitting. \end{corollary}

In fact, $$\dfrac{\sigma_1(L)}{\sigma_2(L)}=\dfrac{\sigma_1(L)}{|\det L|/\sigma_1(L)}=\dfrac{\sigma_1(L)^2}{|\det L|}\ge \dfrac{\lambda_1(L)^2}{|\det L|}=\dfrac{\lambda_1(L)}{\lambda_2(L)},$$ which justifies the previous corollary. Naturally, Theorem \ref{theo3} suggests the following question for fiber buched $GL(d,\R)$-cocycles.



\begin{question}
Let $(T,A)$ be a $GL(d,\R)$-cocycle defined over a transitive subshift of finite type or a transitive Anosov diffeomorphism. Let us suppose that the cocycle $(T,A)$ satisfies the fiber-bunching condition. If there is a constant $\tau>0$ and a full probability set $\mathcal{S}\subset \mathcal{R}$ such that $$\lambda_i(x)- \lambda_{i+1}(x)\ge \tau \quad \textrm{for all $x \in \mathcal{S}$}.$$ Does the cocycle $(T,A)$ have a dominated splitting of index $i$? 
\end{question}

Finally, it would be interesting to prove the existence of $SL(2,\R)$-cocycles with the properties of the example given in Section \ref{example} which almost satisfy the fiber-bunching inequality. More precisely, we would like to prove the following.

\begin{question} \label{question2}
Let $T$ be the left-shift map $T:\{1,2,\dots l\}^{\Z} \to \{1,2,\dots l\}^{\Z}$ defined by $T((x_n)_{n\in \Z})=(x_{n+1})_{n\in \Z}$. Let us consider the metric as before in the Preliminaries. Let $c>1$ be an arbitrary positive constant. Is there an $\alpha$-H\"{o}lder $SL(2,\R)$-cocycle $(T,A)$ which is not uniformly hyperbolic, but there is a constant $\epsilon>0$ such that $\lambda_+(x)>\epsilon$ for every point $x$ in a set $\mathcal{S}$ of full probability, and also $\|A(x)\|^2\cdot 2^{-\alpha}<c \text{?}$

\end{question}

\begin{remark}
If $c$ were less than 1, this would mean that the cocycle $(T,A)$ is fiber-bunched. Our construction works for any positive constant $\alpha$ less than 1 and $\|A(x)\|=2$ for every $x\in X$. Hence, it does not satisfy the requirements of Question $\ref{question2}$.
\end{remark}

\end{document}